\documentclass{amsart}
\usepackage[utf8x]{inputenc}
\usepackage{amsmath, amsthm, amssymb}
\usepackage[usenames,dvipsnames,svgnames,table]{xcolor}
\usepackage[margin=1.3in]{geometry}
\usepackage[french,english]{babel}

\newtheorem{theorem}{Theorem}[section]
\newtheorem{proposition}[theorem]{Proposition}
\newtheorem{lemma}[theorem]{Lemma}
\newtheorem*{remark}{Remark}

\newcommand{\R}{\mathbb R}
\newcommand{\T}{\mathbb T}
\newcommand{\cG}{\mathcal G}
\newcommand{\cE}{\mathcal E}

\newcommand{\dd}{\, \mathrm{d}}

\numberwithin{equation}{section}

\newcommand{\commentout}[1]{}

\title[Velocity Decay Estimates for Boltzmann with Hard Potentials]{Velocity Decay Estimates for Boltzmann equation with hard potentials}
\author{Stephen Cameron}
\address{Courant Institute, New York University, New York, NY 10012}
\email{spc6@cims.nyu.edu}
\author{Stanley Snelson}
\address{Department of Mathematical Sciences, Florida Institute of Technology, Melbourne, FL 32901}
\email{ssnelson@fit.edu}

\begin{document}

\begin{abstract}
We establish pointwise polynomial decay estimates in velocity space for the spatially inhomogeneous Boltzmann equation without cutoff, in the case of hard potentials ($\gamma +2s > 2$), under the assumption that the mass, energy, and entropy densities are bounded above, and the mass density is bounded below. These estimates are self-generating, i.e. they do not require corresponding decay assumptions on the initial data. Our results extend the recent work of Imbert-Mouhot-Silvestre (arXiv:1804.06135), which addressed the case of moderately soft potentials ($\gamma + 2s \in [0,2]$). 
\end{abstract}

\maketitle

\section{Introduction}

The Boltzmann equation is a fundamental model in kinetic theory and statistical physics. It describes the evolution of a particle density of a rarefied gas in phase space. A solution $f: [0,T]\times \T^d\times \R^d\to [0,\infty)$ of the Boltzmann equation satisfies
\begin{equation}\label{e:boltzmann}
\partial_t f(t,x,v) + v\cdot \nabla_xf(t,x,v) = Q(f,f)(t,x,v), 
\end{equation}
where (suppressing the dependence on $t$ and $x$)
\begin{equation}
Q(f,f)(v) = \int\limits_{\R^d}\int\limits_{\mathbb S^{d-1}} (f(v_*')f(v') - f(v_*)f(v))B(|v-v'|, \cos(\theta))  \dd \sigma \dd v_*,
\end{equation}
and  
\begin{equation}
v' = \frac{v+v_*}{2}+\frac{|v-v_*|}{2}\sigma, \qquad v_*' = \frac{v+v_*}{2}-\frac{|v-v_*|}{2}\sigma, \quad \text{ and }
\cos(\theta) : = \frac{v-v_*}{|v-v_*|}\cdot \sigma.
\end{equation}
Note that this implies
\begin{equation}
\sin(\theta/2) = \frac{v'-v}{|v'-v|}\cdot \sigma.
\end{equation}
We are interested in the non-cutoff case where the collision kernel $B$ has a nonintegrable singularity near $\theta = 0$, which reflects the fact that long-range interactions are taken into account. In particular, $B$ takes the form 
\begin{equation}
B(r,\cos(\theta)) = r^\gamma b(\cos(\theta)), \qquad b(\cos(\theta))\sim |\sin(\theta/2)|^{-(d-1)-2s},
\end{equation}
where $\gamma> -d$, and $s\in (0,1)$.  More specifically, we consider the case of \emph{hard potentials} where $\gamma\in (0,2)$ and $\gamma +2s > 2$.

The purpose of this article is to study the \emph{a priori} decay of solutions to \eqref{e:boltzmann} for large velocity, under conditional assumptions that we make precise in \eqref{e:hydrodynamicbounds} below. Decay as $|v|\to \infty$ is an important and ever-present issue in the mathematical study of the Boltzmann equation, because of the need to control the integral operator $Q(f,f)$, which has a growing (or slowly-decaying, depending on $\gamma$) factor of $|v-v'|^\gamma$ in its kernel. 

The question of global existence vs. breakdown for large solutions of \eqref{e:boltzmann} (i.e. solutions that are not necessarily close to an equilibrium state) is a celebrated open problem, and may be out of reach with current techniques. A more realistic, but still highly nontrivial, goal was conjectured in \cite{imbert2016weak}: that solutions are $C^\infty$ provided the following estimates on the hydrodynamic quantities (mass, energy, entropy densities) hold for all $(t,x)\in [0,T]\times \T^d$:
\begin{equation}\label{e:hydrodynamicbounds}
\begin{split}
0<m_0 \leq \int\limits_{\R^d} f(t,x,v)\dd v \leq M_0 <\infty,
\\ \int\limits_{\R^d} f(t,x,v)|v|^2 \dd v \leq E_0<\infty,
\\ \int\limits_{\R^d} f(t,x,v)\log(f(t,x,v)) \dd v \leq H_0 <\infty.
\end{split}
\end{equation}
Roughly speaking, this means any singularity in the Boltzmann equation is physically observable at the macroscopic scale. This conjecture was very recently proven, in the case of moderately soft potentials ($\gamma +2s\in [0,2]$) in \cite{imbert2019smooth}. This followed a number of works in the same direction, starting in 2014:
\begin{itemize}
	\item In \cite{silvestre2016boltzmann}, Silvestre showed that solutions satisfying \eqref{e:hydrodynamicbounds} enjoy an \emph{a priori} bound in $L^\infty$, in the case $\gamma < 2$ and $\gamma + 2s > 0$. (A similar conclusion holds for other ranges of $\gamma$ and $s$, but with extra assumptions on $f$.)
	
	\item In \cite{imbert2016weak}, Imbert-Silvestre improved this conclusion to $C^\alpha$ regularity, via a De Giorgi-type estimate.
	
	\item In \cite{LuisMain}, working in the case of moderately soft potentials ($\gamma \in (-2,2)$ and $\gamma+2s\in [0,2]$), Imbert-Mouhot-Silvestre showed that pointwise polynomial decay in $v$ is propagated forward (i.e. if $f_{\rm in}(x,v) \lesssim |v|^{-q}$, then $f(t,x,v) \lesssim |v|^{-q}$) if $\gamma \in (-2,0]$, and in the case $\gamma \in (0,2)$, this decay is self-generating (i.e. $f(t,x,v) \lesssim t^{-p}|v|^{-q}$ some $p,q>0$, for arbitrary initial data).

	\item In \cite{imbert2018schauder}, Imbert-Silvestre established Schauder-type regularity estimates for linear kinetic equations, and finally, in \cite{imbert2019smooth}, the same authors applied these estimates in a delicate bootstrapping procedure to conclude $C^\infty$ smoothness, still under the assumptions \eqref{e:hydrodynamicbounds} and in the moderately soft case.
\end{itemize}
The present article extends the result of \cite{LuisMain}, specifically the self-generating upper bounds, to the case $\gamma + 2s > 2$. 
We should note that decay estimates for the moderately soft potentials case were a crucial ingredient in establishing the higher regularity of solutions in \cite{imbert2019smooth}, and part of the motivation of our main theorem is to open the door to conditional regularity under the assumption \eqref{e:hydrodynamicbounds} in the case of hard potentials. 
The importance of decay in $v$ can also be seen in the higher regularity theory of the closely related Landau equation, where a certain number of moments in $v$ are ``used up'' in each step of the bootstrapping procedure \cite{henderson2017smoothing}.

The main added difficulty in the regime $\gamma+2s>2$ is that our assumptions on hydrodynamic quantities (pointwise mass, energy, entropy bounds as in \eqref{e:hydrodynamicbounds}) no longer guarantee a finite $\gamma+2s$ moment, which is needed to control the integral kernel arising in the Carleman representation of $Q(f,f)$ (see Section \ref{s:prelim}).  In order to deal with this, we will work in the space of solutions satisfying
\begin{equation}\label{e:Kbound}
K_0:= \sup\limits_{(t,x)\in (0,1]\times \T^d} t^{\frac{\gamma+2s}{\gamma}}\int\limits_{\R^d} f(t,x,v) |v|^{\gamma+2s} \dd v <\infty .
\end{equation}
Without loss of generality, we assume the time of existence $T = 1$. The constant $K_0$ is not controlled by any physically meaningful quantity, so we seek estimates that are independent of $K_0$.

\begin{remark}
	We say that a constant is universal if it only depends on the constants $d$, $\gamma$, $s$, $m_0$, $M_0$, $E_0$, $ H_0$, and the cross section $B(r,\cos(\theta))$.  We use the notation $\lesssim$, $\gtrsim$, and $\approx$ to mean a quantity is bounded from above, below, or equivalent up to a universal constant.  In our calculations, we will keep track of how constants depend on $q$ and $K_0$. 
\end{remark}

 Our main result is
\begin{theorem}\label{t:main}
	There exists a universal constant $q_0>0$ such that, for all $q\geq q_0$ and any classical solution $f$ of the Boltzmann equation \eqref{e:boltzmann} satisfying \eqref{e:hydrodynamicbounds} for all $(t,x) \in [0,1]\times \mathbb T^d$, \eqref{e:Kbound}, and 
\begin{equation}\label{e:velocitydecay}
\lim\limits_{|v|\to \infty} |v|^qf(t,x,v) = 0 \qquad \text{uniformly in } (t,x)\in [0,1]\times \T^d,
\end{equation}	
	there holds
	\[f(t,x,v) \leq C_q t^{-q/\gamma - d/2s} (1+|v|)^{-q}, \quad (t,x,v) \in (0,1]\times \T^d \times \R^d,\]
	where $C_q$ depends only on $q$ and universal constants.
\end{theorem}
In particular, note that $C_q$ is independent of $K_0$ and the rate of decay in \eqref{e:velocitydecay}. 

The proof of Theorem \ref{t:main} also gives a universal bound on the constant $K_0$ in \eqref{e:Kbound}, which implies that $\int f(t,x,v) |v|^{\gamma+2s} \dd v \lesssim t^{-(\gamma+2s)/\gamma}$. By the De Giorgi estimate of \cite{imbert2016weak}, this implies $f$ is H\"older continuous, with (local) H\"older estimates depending only on universal constants. (See Remark 1.4 of \cite{imbert2016weak}.) We expect that the higher regularity argument of \cite{imbert2019smooth} would also extend (perhaps with some adjustments) to the case $\gamma + 2s>2$, via our decay estimates for $f$ and the universal bound on $K_0$.

\subsection{Related work}\label{s:related}

There is a long history of decay estimates for the Boltzmann equation, though mostly for the space homogeneous case (where $f$ is assumed to be independent of $x$): for estimates on $L^1_v$ moments, both polynomial and exponential, see \cite{ikenberry1956kinetic, elmroth1983boltzmann, desvillettes1993moments, wennberg1996povzner, bobylev1997moment, lu1999boltzmann, mischler2006cooling, alonso2013exponential} and the references therein. Pointwise polynomial decay (i.e. $L^\infty_v$ moments) were considered in \cite{carleman1933boltzmann, arkeryd1983boltzmann} for the homogeneous (cutoff) equation, and pointwise exponential decay was established in \cite{gamba2009upper}. The latter result was extended to the non-cutoff homogeneous equation in \cite{gamba2017exponential}. It should be noted that the conditional assumptions \eqref{e:hydrodynamicbounds} are not necessary in the space homogeneous case, because the mass $M_f(t)$ and energy $E_f(t)$ are conserved by the evolution of the equation, and the entropy $H_f(t)$ is nonincreasing.

Fewer decay results of this type are available for the inhomogeneous equation, but in \cite{gualdani2017htheorem}, generation of polynomial moments in $L^\infty_x L^1_v$ was established for the hard spheres collision kernel, under conditional assumptions similar to \eqref{e:hydrodynamicbounds}. 

As mentioned above, unconditional global existence of solutions to the non-cutoff Boltzmann equation with general initial data remains unknown. Global-in-time solutions have only been constructed in special cases such as the close-to-equilibrium \cite{gressman2011boltzmann} and the space homogeneous \cite{desvillettes2009stability} settings.

\subsection{Proof strategy}\label{s:ideas}

Our preliminary goal is to show for all $q>0$ sufficiently large that 
\begin{equation}\label{e:Cq}
f(t,x,v)\leq C_q(K_0)t^{-q/\gamma -d/2s}|v|^{-q},
\end{equation}
where $C_q(K_0)$ depends on $q$ and $K_0$ in an explicit way. (This is Theorem \ref{t:f-upper-bound}.) The proof of this estimate is based on the method of \cite{LuisMain}. Defining $g(t,v) = N_qt^{-q/\gamma - d/2s} |v|^{-q}$ for a constant $N_q$ to be determined, our decay assumption on $f$ implies that $f(t,x,v) < g(t,v)$ for small $t$ (see Section \ref{s:breakthrough}). Then we look at the first crossing point where $f(t,x,v) = g(t,v)$ and seek a contradiction. Since $g$ is independent of $x$, we have $\nabla_x f = 0$ and $\partial_t f \geq \partial_t g$ at this point. One then has
\begin{equation}
\partial_t g \leq \partial_t f  = Q(f,f).
\end{equation}
 It then suffices to show $Q(f,f)<\partial_t g$ at the crossing point to derive a contradiction. Using the Carleman representation (see Section \ref{s:prelim}), we can decompose $Q$ into a fractional diffusion operator $Q_1$ and a lower-order term $Q_2$. Further decomposing the integral $Q_1$ into a ``good'' term $\cG$ and four error terms, we will show that the diffusive part $\cG$ is negative and dominates the other parts of $Q$, leading to the desired inequality $Q(f,f) < \partial_t g$ at the crossing point.


To complete the proof of Theorem \ref{t:main}, we need to show that the quantity $K_0$ can be bounded by universal constants. Using the pointwise estimate \eqref{e:Cq} and our hydrodynamic bounds, we will show that 
\begin{equation}
t^{\frac{\gamma+2s}{\gamma}}\int\limits_{\R^d} f(t,x,v) |v|^{\gamma+2s} \dd v \leq CK_0^{1-\epsilon}.
\end{equation}
Taking the supremum in $(t,x)$, we then have that $K_0^\epsilon \leq C$ for some $C$ universal.  Hence, the pointwise bounds \eqref{e:Cq} will in fact only depend on our original hydrodynamic quantities. This last step is reminiscent of the method applied to the hard potentials case of the Landau equation by the second author \cite{StanMain}, where decay estimates were first established with constants depending on the $\gamma + 2$ moment of $f$, and then this dependence was removed by carefully interpolating between the pointwise decay estimates and the energy density bound.

\subsection{Organization of the paper} In Section \ref{s:prelim}, we review the Carleman representation of $Q(f,f)$ and the $L^\infty$ estimate of \cite{silvestre2016boltzmann} which we will need. In Section \ref{s:breakthrough}, we make precise the breakthrough argument described in Section \ref{s:ideas}.  In Section \ref{s:Q}, we prove the pointwise estimate \eqref{e:Cq} by estimating the collision operator $Q(f,f)$ at a crossing point. In Section \ref{s:proof}, we complete the proof of Theorem \ref{t:main} by showing the constant $K_0$ in \eqref{e:Kbound} is bounded universally.

%
%
%

\section{Preliminaries and known results}\label{s:prelim}

We will use so-called Carleman coordinates to write the bilinear collision operator $Q(f,f)$ as a sum of two terms,
\begin{equation}
\begin{split}
Q(f,f) = Q_1(f,f) + Q_2(f,f),
\end{split}
\end{equation}
where 
\begin{equation}\label{e:Q1Q2}
\begin{split}
Q_1(f,f) &:= \int\limits_{\R^d} \int\limits_{\mathbb S^{d-1}} (f(v') - f(v)) f(v_*') B(|v-v'|,\cos\theta) \dd \sigma \dd v_*,\\
Q_2(f,f) &:= f(v) \int\limits_{\R^d}\int\limits_{\mathbb S^{d-1}} (f(v_*') - f(v_*)) B(|v-v'|,\cos\theta) \dd \sigma \dd v_*.
\end{split}
\end{equation}
(We will routinely suppress the dependence of $f$ on $t$ and $x$ in calculations involving the collision operator, since this operator acts only in velocity space.) These integrals are certainly well-defined, since we are assuming $f$ is $C^2$ and has polynomial decay of high order in $v$. The following results, quoted from \cite{silvestre2016boltzmann}, will allow us to  treat $Q_1$ and $Q_2$ as a fractional diffusion operator and a lower-order term, respectively. To simplify notation, we will write $h = v-v'$ for the remainder of the paper.
\begin{lemma}{\cite[Lemma 4.1 and Corollary 4.2]{silvestre2016boltzmann}}
The term $Q_1(f,f)$ can be written
\begin{equation}
Q_1(f,f)(v) = \int\limits_{\R^d} (f(v+h)-f(v))K_f(v,h)\dd h,
\end{equation}
where the kernel $K_f(v,h)$ is defined by
\begin{equation}\label{e:Kf}
 K_f(v,h) = |h|^{-d-2s}\int\limits_{w\perp h} f(v+w)|w|^{\gamma + 2s+1}A(|h|, |w|) \dd w,
\end{equation}
and $A(|h|,|w|)\approx 1$.
\end{lemma}
The statement of \cite[Corollary 4.2]{silvestre2016boltzmann} is slightly different:
\[ K_f(v,h) \approx |h|^{-d-2s} \int\limits_{w\perp h} f(v+w)|w|^{\gamma+2s+1} \dd w.\]
However, examining the proof shows that, in fact, \eqref{e:Kf} holds with
\[ A(|h|,|w|) = 2^{d-1}(|h|^2+|w|^2)^{(-d+2+\gamma)/2} |h|^{d+2s-1} |w|^{-(\gamma+2s+1)} b(\cos\theta) \approx 1,\]
where $\cos(\theta/2) = |w|/ \sqrt{|h|^2+|w|^2}$.

Next, for the lower order term $Q_2(f,f)$, we have
\begin{lemma}{\cite[Lemmas 5.1 and 5.2]{silvestre2016boltzmann}}
The term $Q_2(f,f)$ can be written
\begin{equation}
Q_2(f,f)(v) = f(v)\left(c_{d,s}\int\limits_{\R^d} f(v+w) |w|^{\gamma} \dd w \right) = f(v)(c_{d,s} f * |w|^\gamma)(v),
\end{equation}  
where $c_{d,s}>0$ depends only on $d$ and $s$.
\end{lemma}


We will also need the main conclusion of \cite{silvestre2016boltzmann}, which is an \emph{a priori} bound for solutions in $L^\infty(\T^d\times \R^d)$ for each $t>0$:
\begin{proposition}{\cite[Theorem 1.2]{silvestre2016boltzmann}}\label{p:Linfty}
Let $f$ be a classical solution of \eqref{e:boltzmann} on $[0,1]\times \T^d\times \R^d$ that satisfies \eqref{e:hydrodynamicbounds}. Then there holds
	\[f(t,x,v)\leq N_0 t^{-d/2s},\]
	for some $N_0$ depending only on universal constants.
\end{proposition}

See also \cite[Theorem 4.1]{LuisMain} for another statement of this result, with the time dependence of the right-hand side given explicitly.

\section{Breakthrough Argument}\label{s:breakthrough}

Fix some $q>d+\gamma+2s+1$, which we will later take to be sufficiently large, and a classical solution $f:[0,1]\times \T^d\times \R^d \to [0,\infty)$ to \eqref{e:boltzmann} satisfying hydrodynamic bounds \eqref{e:hydrodynamicbounds}, finite $\gamma +2s$ moment \eqref{e:Kbound}, and uniform polynomial decay as $|v|\to \infty$ \eqref{e:velocitydecay}.  That is, 
\begin{equation}
\lim\limits_{|v|\to \infty} |v|^qf(t,x,v) = 0 \qquad \text{uniformly in } (t,x)\in [0,1]\times \T^d.
\end{equation}

Let $g(t,v) = N_qt^{-q/\gamma - d/2s} |v|^{-q}$.  Then for $t>0$ sufficiently small, we know by \eqref{e:velocitydecay} that $f(t,x,v)<g(t,v)$.  As discussed above, our goal is to show that $f(t,x,v)\leq g(t,v)$ for some $N_q$ depending only on $K_0$ and universal constants.  

Suppose not.  Then since $f$ is periodic in $x$ and decays in $v$, there must be a first crossing point.   That is, there is some $(t,x,v)\in (0,1]\times \T^d\times \R^d$ such that 
\begin{equation}\label{e:touchabove}
\left\{ \begin{array}{ll}
f(t,x,v) \ =g(t,v), &  \\ f(s,y,w)\leq g(s,w), & \forall (s,y,w)\in (0,t]\times \T^d\times \R^d \end{array}\right.
\end{equation}
In particular, we have that $\nabla_xf(t,x,v) = \nabla_x g(t,v) = 0$ and that 
\begin{equation}\label{e:contradiction}
\partial_t g(t,v)\leq \partial_t f(t,x,v) = Q(f,f)(t,x,v).
\end{equation}

Our goal then is to show that under the assumptions of \eqref{e:touchabove} that in fact
\begin{equation}\label{e:Qlessg}
Q(f,f)(t,x,v)<\partial_t g(t,v),
\end{equation}
contradicting \eqref{e:contradiction}.

\section{Estimating $Q(f,f)$}\label{s:Q}

To show that \eqref{e:Qlessg} holds under the assumptions \eqref{e:touchabove}, we need to analyze the collision kernel $Q(f,f)$ in detail. Starting with the Carleman representation \eqref{e:Q1Q2}, we see that the term $Q_1$ is a nonlinear, nonlocal diffusive term of order $2s$, while $Q_2$ is a lower order convolution term.  Since $f$ is being touched from above by $g$, we should expect $Q_1(f,f)(t,x,v) < 0$.  We should be able to control the size of $Q_2$ as it is lower order, so that $Q(f,f)(t,x,v)<0$ on the whole.  So, we begin by investigating the diffusive term.  

We have that 
\begin{equation}
\begin{split}
Q_1(f,f)(v) &= \int\limits_{\R^d}  (f(v+h)-f(v)) |h|^{-d-2s} \int\limits_{w\perp h} f(v+w)|w|^{\gamma + 2s+1}A(|h|, |w|) \dd w \dd h
\\& =  \int\limits_{\R^d}   f(v+w)|w|^{\gamma + 2s} \int\limits_{h\perp w}(f(v+h)-f(v)) |h|^{-d+1-2s}A(|h|, |w|) \dd h \dd w
\\& = \cG(v) + \cE_1(v) + \cE_2(v) + \cE_3(v) + \cE_4(v)
\end{split}
\end{equation}
where 
\begin{equation}\label{e:goodtermdefn}
\cG(v) = \int\limits_{|v+w|\leq |v|/\sqrt{2q+4}} f(v+w)|w|^{\gamma+2s} \int\limits_{h\perp w} (f(v+h)-f(v)) A(|h|,|w|)|h|^{-d+1-2s} \dd h \dd w,
\end{equation}
\begin{equation}\label{e:badterm1defn}
\cE_1(v) = \int\limits_{|v+w|\geq |v|/\sqrt{2q+4}} f(v+w)|w|^{\gamma+2s} \int\limits_{h\perp w,\, |h|\leq |v|/q} (f(v+h)-f(v)) A(|h|,|w|)|h|^{-d+1-2s} \dd h \dd w
\end{equation}
\begin{equation}\label{e:badterm2defn}
\cE_2(v) = \int\limits_{|v|/\sqrt{2}\geq |v+w|\geq |v|/\sqrt{2q+4}} f(v+w)|w|^{\gamma+2s} \int\limits_{h\perp w, \, |h|\geq |v|/q} (f(v+h)-f(v)) A(|h|,|w|)|h|^{-d+1-2s} \dd h \dd w
\end{equation}
\begin{equation}\label{e:badterm3defn}
\cE_3(v) = \int\limits_{ |h|\geq |v|/q, \, |v+h|>|v|/q} (f(v+h)-f(v)) |h|^{-d-2s}\int\limits_{w\perp h, \, |v+w|\geq |v|/\sqrt{2}} f(v+w)|w|^{\gamma+2s+1}  A(|h|,|w|)\dd w \dd h
\end{equation}
\begin{equation}\label{e:badterm4defn}
\cE_4(v) = \int\limits_{ |v+h|\leq |v|/q} (f(v+h)-f(v)) |h|^{-d-2s}\int\limits_{w\perp h, \, |v+w|\geq |v|/\sqrt{2}} f(v+w)|w|^{\gamma+2s+1}  A(|h|,|w|) \dd w \dd h
\end{equation}

The term $\cG(v)$ is the one good term, the source of all the diffusive behavior of $Q_1$.  We are able to estimate it using the concavity of $g$ in the nonradial directions.  The other $\cE_i(v)$ are error terms, which we need to bound from above.  They are split according to how we will be bounding each of them.  $\cE_1(v)$ is a local error term, which we will bound with the local regularity of $g$ near $v$.  $\cE_2(v)$ is the low growth error term.  We will be be able to bound it because over the range of $h$ that we integrate over, $g(v+h)$ will be proportional to $g(v)$.  $\cE_3(v)$ is the low mass error term.  We will bound the decay in the $\gamma+2s$ moment using $g$.  And finally, $\cE_4(v)$ is the most delicate error term, as it will have the same decay in $|v|$ as the good term $\cG(v)$.  But keeping careful track of the dependence on $q$ will ensure that $\cE_4$ is controlled by $\cG(v)$ for $q$ large.

\begin{remark}
A similar decomposition can be made with the Landau equation.  Since the equation is now of order 2 and ``local", only the good term $\cG$ and the local error term $\cE_1$ would remain.  Without these other error terms, one can improve the upper bound on $N_q$ to $\sim q^{q/2}$, which would imply the Gaussian decay proven in \cite{LandauUs}, \cite{StanMain}.
\end{remark}

\subsection{Bounds on Good Term $\cG(v)$}

Our first goal is to bound $\cG(v)$.  This will be the one and only negative term, and the source of all the diffusion.  

\begin{lemma}\label{l:Gbound}
For $f\geq 0$ satisfying the bounds \eqref{e:hydrodynamicbounds}, and $q > d+\gamma+2s+1$. If $(t,x,v)$ is a breakthrough point as in \eqref{e:touchabove}, and $|v|\geq R_q \approx q^{1/2}$, then
\[ \cG(v) \lesssim - q^s |v|^\gamma g(v),\]
where $g(t,x,v) = N_q t^{-q/\gamma-d/2s}|v|^{-q}$, and $\cG$ is defined by \eqref{e:goodtermdefn}.
\end{lemma}
\begin{proof}

Recall that 
\begin{equation}
\cG(v) = \int\limits_{|v+w|\leq |v|/\sqrt{2q+4}} f(v+w)|w|^{\gamma+2s} \int\limits_{h\perp w} (f(v+h)-f(v)) A(|h|,|w|)|h|^{-d+1-2s} \dd h \dd w.
\end{equation}
We first want to bound the inner integral in $h$,
\begin{equation}
 \int\limits_{h\perp w} (f(v+h)-f(v)) A(|h|,|w|)|h|^{-d+1-2s} \dd h,
\end{equation}
from above.  
Note that as the above integral is symmetric in $h$, it suffices to bound 
\begin{equation}
 \int\limits_{h\perp w} (f(v+h)+f(v-h)-2f(v)) A(|h|,|w|)|h|^{-d+1-2s} \dd h.
\end{equation}
Since $g$ touches $f$ from above at $v$, we have that 
\[f(v+h)+f(v-h)-2f(v) \leq g(v+h)+g(v-h) -2g(v).\]

Next, we use the fact that, since $|v+w|$ is small, $w \sim -v$.  Thus, for any $h\perp w$, we should have $h\cdot v \sim 0$.  Explicitly, if $w = z-v$ for some $z\in B_{|v|/\sqrt{2q+4}}$, then $$h\cdot w = 0 \quad\Rightarrow  \quad |h\cdot v| = |h\cdot z| \leq |h|\frac{|v|}{\sqrt{2q+4}}.$$

We claim that 
\begin{equation}\label{e:g-leq-0}
g(v+h)+g(v-h)-2g(v)\leq 0, \qquad \text{for } |h\cdot v|\leq \frac{ |h| |v| }{\sqrt{2q+4}}.
\end{equation}
Recall $g(t,v) = N_q t^{-q/\gamma - d/2s} |v|^{-q}$. 
Let $r = |h|$ and $\theta = \displaystyle\frac{h\cdot v}{|h| \ |v|}$. If $r \geq 2|v||\theta|$, we have $|v\pm h|^2 -|v|^2 = |h|(|h|\pm 2\theta |v|) \geq 0$, and \eqref{e:g-leq-0} follows.

 If $r \leq 2|v||\theta|$, then doing a Taylor expansion of $u(s) = (1+s)^{-q/2}$, we see that
\begin{equation}\label{e:taylorexpansion}
\begin{split}
|v+h|^{-q}+|v-h|^{-q} - 2|v|^{-q} &= |v|^{-q}\left[\left(1+2\theta \frac{r}{|v|} + \frac{r^2}{|v|^2}\right)^{-q/2} + \left(1-2\theta \frac{r}{|v|} + \frac{r^2}{|v|^2}\right)^{-q/2}-2\right]
\\& = |v|^{-q}\left[ -q \frac{r^2}{|v|^2} + q(q+2)\left(\theta^2 \frac{r^2}{|v|^2}+\frac {r^4}{4|v|^4}\right) + \mathcal R \right],
          \end{split}
\end{equation}
%
where the remainder term $\mathcal R< 0$ because $u'''(s) <0$ for $s>-1$. (Note that $\pm 2\theta r/|v| + r^2/|v|^2 \geq -4\theta^2 > -1$.) Since $r^2/|v|^2 \leq 4\theta^2$ and $2\theta^2\leq 1/(q+2)$, we have
\[ - q \frac {r^2}{|v|^2} + q(q+2)\left(\theta^2 \frac {r^2}{|v|^2}  + \frac {r^4}{4|v|^4}\right) \leq q \frac {r^2}{|v|^2} \left( -1 + 2(q+2)\theta^2\right) \leq 0,\]
and \eqref{e:g-leq-0} holds in this case as well.

In fact, for certain $r$ and $\theta$, the quantity $g(v+h) - g(v-h) - 2g(v)$ has a coercive (negative) upper bound. More precisely, the expansion \eqref{e:taylorexpansion} shows
\[|v+h|^{-q} + |v-h|^{-q} - 2|v|^{-q} \leq \left(\frac {2-q}{2q}\right) |v|^{-q} q \frac{r^2}{|v|^2}, \qquad \text{ for } \theta^2 \leq \frac{1}{4q}, \quad 0<r \leq \frac{|v|}{2\sqrt{q}}.\]
Since $q> d+1\geq 3$, this implies
%
\begin{equation}\label{e:g-lowerbound}
g(v+h)+g(v-h)-2g(v)\lesssim  -g(v)q\frac{r^2}{|v|^2}, \qquad \text{ for } (h\cdot v)^2 \leq \frac{|v|^2 |h|^2}{4q}, \quad 0< |h| \leq \frac{|v|}{2\sqrt{q}}.
\end{equation}
By \eqref{e:g-leq-0} and \eqref{e:g-lowerbound}, we have that 
\begin{equation}
\int\limits_{h\perp w} (f(v+h)-f(v)) A(|h|,|w|)|h|^{-d+1-2s} \dd h\leq 0, \qquad \text{for } |v+w|\leq \frac{|v|}{\sqrt{2q+4}},
\end{equation}
with 
\begin{equation}
\begin{split}
\int\limits_{h\perp w} (f(v+h)-f(v)) A(|h|,|w|)|h|^{-d+1-2s} \dd h&\lesssim -\frac{q}{|v|^2}g(v)\int\limits_0^{|v|/2\sqrt{q}} r^{1-2s}\dd r 
\\& \lesssim -q^{s}|v|^{-2s}g(v), 
\end{split}
\end{equation}
whenever $|v+w|\leq \displaystyle\frac{|v|}{2\sqrt{q}}$.  
Thus 
\begin{equation}
\begin{split}
\cG(v)& \lesssim -q^s|v|^{-2s}g(v)\int\limits_{|v+w|\leq |v|/2\sqrt{q}} f(v+w)|w|^{\gamma+2s}\dd w 
\\&\lesssim -q^s |v|^\gamma g(v)\int\limits_{B_{|v|/(2\sqrt{q})}} f(z)\dd z.
\end{split}
\end{equation}

Now we use that, because of the energy bound in \eqref{e:hydrodynamicbounds}, most of the mass of $f$ lies within $B_{|v|/(2\sqrt{q})}$ for $|v|$ large enough. If $|v| \geq R_q :=  \displaystyle\sqrt{\frac{8qE_0}{m_0}}$, then 
\begin{equation}
\int\limits_{\R^d\setminus B_{|v|/(2\sqrt{q})}} f(z)\dd z \leq \frac{4q}{|v|^2}\int\limits_{\R^d} f(z)|z|^2 \dd z \leq \frac{m_0}{2}.
\end{equation}
Hence 
\begin{equation}
\int\limits_{B_{|v|/(2\sqrt{q})}} f(z)\dd z  = \int\limits_{\R^d}f(z)\dd z-\int\limits_{\R^d\setminus B_{|v|/(2\sqrt{q})}} f(z)\dd z \leq \frac{m_0}{2},
\end{equation}
so 
\begin{equation}\label{e:Gbound}
\cG(v) \lesssim -q^s|v|^\gamma g(v),
\end{equation}
as desired.
\end{proof}

\subsection{Bounds on local error term $\cE_1(v)$}
\smallskip

The first error term we will bound is $\cE_1(v)$, the local error term.

\begin{lemma}\label{l:E1bound}
With $f$, $g$, and $q$ as in Lemma \ref{l:Gbound}, and $(t,x,v)$ a breakthrough point with $|v|\geq 1$, we have
\[ \cE_1(v) \lesssim K_0t^{-\frac{\gamma+2s}{\gamma}}q^{\gamma/2 + 3s} |v|^{-2s}g(v),\]
where $K_0$ is the constant from \eqref{e:Kbound}, and $\cE_1(v)$ is defined in \eqref{e:badterm1defn}.
\end{lemma}
\begin{proof}
From the Taylor expansion \eqref{e:taylorexpansion}, we have that 
\begin{equation}
|v+h|^{-q}+|v-h|^{-q} - 2|v|^{-q} \lesssim |v|^{-q}q^2\frac{|h|^2}{|v|^2}, \qquad \text{ for } 0\leq |h|\leq \frac{|v|}{q}.
\end{equation}
Hence, using again that $g$ touches $f$ from above at $v$, we can bound the inner integral in $h$ as 
\begin{equation}
\begin{split}
\int\limits_{h\perp w \ |h|\leq |v|/q} (f(v+h)+f(v-h)-2f(v)) A(|h|,|w|)|h|^{-d+1-2s} \dd h &\lesssim  q^2 |v|^{-2}g(v)\int\limits_{0}^{|v|/q} r^{1-2s}\dd r 
 \\& \lesssim q^{2s}|v|^{-2s}g(v).
\end{split}
\end{equation}
Hence, 
\begin{equation}\label{e:E1bound}
\begin{split}
\cE_1(v) & \lesssim q^{2s}|v|^{-2s} g(v)\int\limits_{|v+w|\geq |v|/\sqrt{2q+4}} f(v+w)|w|^{\gamma+2s} \dd w
\\& \lesssim q^{2s} |v|^{-2s} g(v)(1+\sqrt{2q+4})^{\gamma+2s}\int\limits_{\R^d\setminus B_{|v|/\sqrt{2q+4}}}f(z)|z|^{\gamma+2s}\dd z 
\\& \lesssim K_0t^{-\frac{\gamma+2s}{\gamma}}q^{\gamma/2 + 3s} |v|^{-2s}g(v),
\end{split}
\end{equation}
 where we have used $|v|\leq \sqrt{2q+4}|z|$ in the second line and \eqref{e:Kbound} in the third line.
\end{proof}

\subsection{Bounds on low growth error term $\cE_2(v)$} 
\smallskip


\begin{lemma}\label{l:E2bound}
With $f$, $g$, and $q$ as in Lemma \ref{l:E1bound}, $K_0$ as in \eqref{e:Kbound}, and $(t,x,v)$ a breakthrough point with $|v|\geq 1$, we have
\[\cE_2(v) \lesssim K_0 t^{-\frac{\gamma+2s}{\gamma}} 2^q |v|^{-2s} g(v),\]
where $\cE_2(v)$ is defined by \eqref{e:badterm2defn}.
\end{lemma}
\begin{proof}
Since we restrict $w$ so that $|v+w|\leq \displaystyle\frac{|v|}{\sqrt{2}}$, this forces $h\perp w$ to be such that $$|h\cdot v|\leq \frac{|h| \ |v|}{\sqrt{2}}.$$
Hence for any such $h$, 
$$|v+h|^2 = |v|^2 + 2v\cdot h + |h|^2 \geq |v|^2 - \sqrt{2}|v| \ |h| + |h|^2 \geq \frac{|v|^2}{2}.$$
Thus, since $f(v)\geq 0$ and $g(t,v) = N_q t^{-q/\gamma-d/2s} |v|^{-q}$,
\begin{equation}
f(v+h)-f(v)\leq g(v+h)\leq 2^{q/2} g(v), \qquad |v+w|\leq \frac{|v|}{\sqrt{2}}, \ h\perp w.
\end{equation}
Therefore, we can bound the inner integral in $h$ by 
\begin{equation}
\begin{split}
 \int\limits_{h\perp w, \ |h|\geq |v|/q} (f(v+h)-f(v)) A(|h|,|w|)|h|^{-d+1-2s} \dd h &\lesssim 2^{q/2}g(v)\int\limits_{|v|/q}^\infty r^{-1-2s} \dd r 
 \\& \lesssim 2^{q/2} q^{2s}|v|^{-2s}g(v).
 \end{split}
\end{equation}
 As in the proof of Lemma \ref{l:E1bound}, it follows that the total error term can be bounded as 
\begin{equation}\label{e:E2bound}
\begin{split}
\cE_2(v) & \lesssim 2^{q/2} q^{2s}|v|^{-2s}g(v)\int\limits_{|v|/\sqrt{2}\geq |v+w|\geq |v|/\sqrt{2q+4}} f(v+w)|w|^{\gamma+2s} \dd w 
\\& \lesssim 2^{q/2}q^{2s} |v|^{-2s} g(v)(1+\sqrt{2q+4})^{\gamma+2s}\int\limits_{\R^d\setminus B_{|v|/\sqrt{2q+4}}}f(z)|z|^{\gamma+2s} \dd z 
\\& \lesssim K_0t^{-\frac{\gamma+2s}{\gamma}}2^{q/2}q^{\gamma/2 + 3s} |v|^{-2s}g(v)
\\&\lesssim K_0t^{-\frac{\gamma+2s}{\gamma}}2^{q}|v|^{-2s}g(v).
\end{split}
\end{equation}

\end{proof}

\subsection{Bounds on low mass error term $\cE_3(v)$}


\begin{lemma}\label{l:E3bound}
With $f$, $g$, and $q$ as in Lemma \ref{l:E1bound}, $K_0$ as in \eqref{e:Kbound}, and $(t,x,v)$ a breakthrough point with $|v|\geq 1$, we have
\[ \cE_3(v) \lesssim K_0t^{-\frac{\gamma+2s}{\gamma}}2^{q}|v|^{-2s} g(v),\]
where $\cE_3(v)$ is defined by \eqref{e:badterm3defn}.
\end{lemma}
\begin{proof}
We first have to bound the inner integral in $w$.  By bounding $f$ with $g$ and noting that $|w|\leq |v+w|+|v|\lesssim |v+w|$, we obtain 
\begin{equation}
\int\limits_{w\perp h, \ |v+w|\geq |v|/\sqrt{2}} f(v+w)|w|^{\gamma+2s+1}  A(|h|,|w|)\dd w \lesssim \int\limits_{w\perp h, \ |v+w|\geq |v|/\sqrt{2}} g(v+w)|v+w|^{\gamma+2s+1}\dd w
\end{equation}
Taking $z = v+w$, we thus need to bound 
\begin{equation}
\int\limits_{(z-v)\perp h, \ |z|\geq |v|/\sqrt{2}} g(z)|z|^{\gamma+2s+1} \dd z.  
\end{equation}
As $(z-v)\perp h$, we have that $z\cdot h = v\cdot h$.  Let $z^\perp$ be the portion of $z$ perpendicular to $h$ and $\hat{h}=\displaystyle\frac{h}{|h|}$.  Then, letting $q' = q-\gamma-2s-1>d$, we see that
\begin{equation}
\begin{split}
\int\limits_{(z-v)\perp h, \ |z|\geq |v|/\sqrt{2}} |z|^{-q'}\dd z &= \int\limits_{|z|\geq |v|/\sqrt{2}} ((\hat{h}\cdot v)^2+|z^\perp|^2)^{-q'/2}\dd z^\perp 
\\& = \int\limits_{\sqrt{(|v|^2/2 - (\hat{h}\cdot v)^2)_+}}^\infty ((\hat{h}\cdot v)^2 + r^2)^{-q'/2}r^{d-2}\dd r.
\end{split}
\end{equation}
In the case that $(\hat{h}\cdot v)^2 \geq |v|^2/2$, we can bound this easily as 
\begin{equation}
\begin{split}
\int\limits_{0}^\infty ((\hat{h}\cdot v)^2 + r^2)^{-q'/2}r^{d-2}\dd r &= |\hat{h}\cdot v|^{-q'+d-1}\int\limits_0^\infty (1+(r')^2)^{-q'/2} (r')^{d-2} \dd r'
\\& \lesssim 2^{q'/2}|v|^{-q'+d-1}
\\&\lesssim 2^{q/2}|v|^{-q+\gamma+2s+d},
\end{split}
\end{equation}
where the integral in $r'$ is $\lesssim 1$ since $q' > d$.

If $|\hat{h}\cdot v|^2 < |v|^2/2$, then by shifting the integral we can similarly get that 
\begin{equation}
\begin{split}
\int\limits_{\sqrt{|v|^2/2 - (\hat{h}\cdot v)^2}}^\infty &((\hat{h}\cdot v)^2 + r^2)^{-q'/2}r^{d-2}\dd r
\\& = \int\limits_0^\infty \left(|v|^2/2 + 2r\sqrt{|v|^2/2 - (\hat{h}\cdot v)^2} + r^2\right)^{-q'/2} \left(r+\sqrt{|v|^2/2 - (\hat{h}\cdot v)^2}\right)^{d-2}\dd r 
\\& \lesssim \int\limits_0^\infty (|v|^2/2 +r^2)^{-(q'-d+2)/2} \dd r 
\\& \lesssim 2^{q'/2}|v|^{-q'+d-1}
\\&\lesssim 2^{q/2}|v|^{-q+\gamma+2s+d}.
\end{split}
\end{equation}
Thus in either case, recalling the formula $g(t,v) = N_q t^{-q/\gamma-d/2s}|v|^{-q}$, we have that 
\begin{equation}
\int\limits_{w\perp h, |v+w|\geq |v|/\sqrt{2}} f(v+w)|w|^{\gamma+2s+1}  A(|h|,|w|)\dd w \lesssim 2^{q/2}|v|^{\gamma+2s+d}g(v).
\end{equation}

Hence, we can bound our third error term as 
\begin{equation}\label{e:E3bound}
\begin{split}
\cE_3(v)&\lesssim 2^{q/2}|v|^{\gamma+2s+d}g(v)\int\limits_{ |h|\geq |v|/q, |v+h|>|v|/q} (f(v+h)-f(v)) |h|^{-d-2s} \dd h 
\\& \lesssim 2^{q/2}q^{2s+d}|v|^{\gamma}g(v)\int\limits_{ |h|\geq |v|/q, |v+h|>|v|/q} f(v+h)\dd h
\\& \lesssim 2^{q/2}q^{\gamma +4s+d} |v|^{-2s}g(v)\int\limits_{\R^d\setminus B_{|v|/q}} f(z)|z|^{\gamma+2s}\dd z 
\\& \lesssim K_0t^{-\frac{\gamma+2s}{\gamma}} 2^{q/2}q^{\gamma +4s+d} |v|^{-2s}g(v)
\\& \lesssim K_0t^{-\frac{\gamma+2s}{\gamma}}2^{q}|v|^{-2s} g(v),
\end{split}
\end{equation}
 using \eqref{e:Kbound}. Here, it is necessary to use the bound on the $\gamma+2s$ moment of $f$ so that the estimate of $\cE_3(v)$ has sufficient decay for large $|v|$.
\end{proof}

\subsection{Bound on the delicate error term $\cE_4(v)$}
%
%

\begin{lemma}\label{l:E4bound}
With $f$, $g$, and $q$ as in Lemma \ref{l:E1bound} and $(t,x,v)$ a breakthrough point with $|v|\geq 1$, we have
\[ \cE_4(v) \lesssim |v|^\gamma g(v),\]
where $\cE_4(v)$ is defined by \eqref{e:badterm4defn}.
\end{lemma}

\begin{proof}
Again, we need to first bound the inner integral.  As $|v+h|\leq \displaystyle\frac{|v|}{q}$, we have that $$\displaystyle\frac{|v\cdot h|}{|v| \ |h|} \geq \sqrt{1-\frac{1}{q^2}}\geq \frac{1}{\sqrt{2}}.$$  Hence, using $f\leq g$ as in Lemma \ref{l:E3bound}, we can bound the inner integral as we did for $\cE_3(v)$ as 
\begin{equation}
\begin{split}
 \int\limits_{(z-v)\perp h, \ |z|\geq |v|/\sqrt{2}} |z|^{-q'}\dd z &= \int\limits_{|z|\geq |v|/\sqrt{2}} ((\hat{h}\cdot v)^2+|z^\perp|^2)^{-q'/2}\dd z^\perp 
\\& = \int\limits_{0}^\infty ((\hat{h}\cdot v)^2 + r^2)^{-q'/2}r^{d-2}\dd r
\\& \leq |\hat{h}\cdot v|^{-q+\gamma+2s+d} \int\limits_0^\infty (1+(r')^2)^{-q'/2} (r')^{d-2} \dd r'.
\end{split}
\end{equation}
We need to be careful about how we bound this expression, since a factor $2^{q/2}$ in this term would ruin our estimates later on.  As $|v+h|\leq\displaystyle\frac{|v|}{q}$, we have that
\begin{equation}
|\hat{h}\cdot v|^{-q}  \leq\left(\sqrt{1-\frac{1}{q^2}}\right)^{-q}  |v|^{-q}\lesssim |v|^{-q}.
\end{equation}
Furthermore, 

	\begin{equation}
\int\limits_0^\infty (1+(r')^2)^{-q'/2} (r')^{d-2} \dd r' \lesssim 1.
\end{equation}
Hence, 
\begin{equation}
\int\limits_{w\perp h, |v+w|\geq |v|/\sqrt{2}} f(v+w)|w|^{\gamma+2s+1}  A(|h|,|w|)\dd w\lesssim |v|^{\gamma+2s+d}g(v)
\end{equation}
Since $|h|\approx |v|$ whenever $|v+h|\leq \displaystyle\frac{|v|}{q}$, we thus have that 
\begin{equation}\label{e:E4bound}
\begin{split}
\cE_4(v) &\lesssim |v|^{\gamma+2s+d} g(v)\int\limits_{|v+h|\leq |v|/q} f(v+h)|h|^{-d-2s}\dd h 
\\& \lesssim |v|^{\gamma} g(v),
\end{split}
\end{equation}
as desired.
\end{proof}

\subsection{Bounding the convolution $Q_2(f,f)(v)$}

Recall that 
\begin{equation}
Q_2(f,f)(v) = c_{d,s}f(v)\int\limits_{\R^d} f(v+w)|w|^\gamma \dd w.  
\end{equation}
Since $\gamma>0$, we can bound $|w|^\gamma$ as $$|w|^\gamma \lesssim |v|^\gamma+ |v+w|^{\gamma}.$$
Thus as $\gamma \leq 2$, we can bound this by our mass and energy bounds  
\begin{equation}
\int\limits_{\R^d} f(v+w)|w|^\gamma \dd w \lesssim |v|^\gamma \int\limits_{\R^d} f(z) \dd z + \int\limits_{\R^d} f(z)|z|^\gamma \dd z \lesssim (1+|v|^\gamma).
\end{equation}
This implies the following simple estimate:
\begin{lemma}\label{l:Q2bound}
At a breakthrough point $(t,x,v)$ such that $|v|\geq 1$, we have
\[Q_2(f,f)(v)\lesssim |v|^\gamma g(v).\]
\end{lemma}



%
%

\subsection{Completing the estimate on $Q(f,f)$}

\begin{theorem}\label{t:f-upper-bound}
 There exists $q_0>0$ universal, such that for all $q\geq q_0$, there holds
\begin{equation}\label{e:fbound}
f(t,x,v)\leq N_q t^{-q/\gamma -d/2s}|v|^{-q}, \qquad  (t,x,v)\in (0,1]\times \T^d\times \R^d,
\end{equation}
with
\begin{equation}\label{e:Nq}
N_q = \max\left\{C^qK_0^{q/(\gamma+2s)} 2^{q^2/(\gamma+2s)},  C^{q/\gamma}\left(\frac{q}{\gamma}+\frac{d}{2s}\right)^{q/\gamma}\right\},
\end{equation}
and $C>0$ universal.
\end{theorem}
\begin{proof}
First, we claim that a crossing point must occur at large velocity: more precisely, at any crossing point $(t,x,v)$, we must have
\begin{equation}\label{e:vlowerbound2}
|v|^{\gamma+2s}\gtrsim K_0t^{-\frac{\gamma+2s}{\gamma}} 2^q .
\end{equation}
To guarantee this, we need to take advantage of the $L^\infty$ bound of \cite{silvestre2016boltzmann}, quoted above as Proposition \ref{p:Linfty}, which gives
\begin{equation}
f(t,x,v)\leq N_0 t^{-d/2s},
\end{equation}
Recall that $g(t,v) = N_q t^{-q/\gamma -d/2s}|v|^{-q}$.  At a touching point, we must necessarily have 
\begin{equation}
N_q t^{-q/\gamma -d/2s}|v|^{-q} \leq N_0 t^{-d/2s} 
\end{equation}
Hence, we get the following lower bound on $|v|$:
\begin{equation}\label{e:vlowerbound}
|v|\geq \left(\frac{N_q}{N_0}\right)^{1/q}t^{-1/\gamma}, 
\end{equation}
Thus by choosing 
\begin{equation}
\frac{N_q}{N_0} \geq   C^qK_0^{q/(\gamma+2s)} 2^{q^2/(\gamma+2s)}, 
\end{equation}
for some $C$ universal, we can guarantee that \eqref{e:vlowerbound2} holds.

Next, since $|v|\gtrsim 2^{q/(\gamma+2s)} \gtrsim \sqrt q$, the combination of Lemmas \ref{l:Gbound}, \ref{l:E1bound}, \ref{l:E2bound}, \ref{l:E3bound}, \ref{l:E4bound}, and \ref{l:Q2bound} gives
\begin{equation}
\left\{\begin{array}{cl} \cG(v)& \lesssim -q^s|v|^\gamma g(v), \\ \cE_1(v)+\cE_2(v)+\cE_3(v) & \lesssim K_0t^{-\frac{\gamma+2s}{\gamma}}2^{q}|v|^{-2s}g(v), \\ \cE_4(v)+Q_2(f,f)(v) & \lesssim |v|^\gamma  g(v),\end{array}\right.
\end{equation}
at any crossing point. By taking $q$ sufficiently large (universal), we can thus guarantee that 
\begin{equation}
\cG(v)+\cE_4(v) +Q_2(f,f)(v)\lesssim -q^s|v|^\gamma g(v).
\end{equation}
The remaining terms $\cE_1(v) + \cE_2(v) + \cE_3(v)$, can be bounded using the lower bound \eqref{e:vlowerbound2} for $|v|$, and we have shown 
%
\begin{equation}\label{e:Qff}
Q(f,f)(t,x,v)< -C^{-1}q^s|v|^\gamma g(t,v),
\end{equation}
for some $C$ universal.

Our choice of $N_q$ also implies 
\begin{equation}
\frac{N_q}{N_0} \geq  C^{q/\gamma}\left(\frac{q}{\gamma}+\frac{d}{2s}\right)^{q/\gamma},
\end{equation}
which, in combination with \eqref{e:vlowerbound}, guarantees that 
\begin{equation}
|v|^\gamma \geq C\left(\frac{q}{\gamma}+\frac{d}{2s}\right)t^{-1}.
\end{equation}
Hence, via \eqref{e:Qff}, we have
\begin{equation}
Q(f,f)(v)< -C^{-1}q^s|v|^\gamma g(v)\leq -\left(\frac{q}{\gamma}+\frac{d}{2s}\right)t^{-1}g(t,v) = \partial_t g(t,v).
\end{equation}
Thus, by the breakthrough argument of Section \ref{s:breakthrough}, we see that taking $N_q$ as in \eqref{e:Nq} for $q$ sufficiently large 
guarantees the upper bound \eqref{e:fbound} for $f$.
\end{proof}

\section{Bound on $K_0$ and conclusion of the proof}\label{s:proof}

All that remains is to show that $K_0 = \sup_{t,x} \int_{\R^d} f(t,x,v) |v|^{\gamma+2s}\dd v$ is in fact bounded by a universal constant to complete the proof of our main theorem.  

\begin{proof}[Proof of Theorem \ref{t:main}]
First, we can assume without loss of generality that for all $q\geq q_0$,
\begin{equation}\label{e:Nq2}
N_q = C^qK_0^{q/(\gamma+2s)} 2^{q^2/(\gamma+2s)}.
\end{equation}
Indeed, otherwise there holds, for some $q\geq q_0$,
\begin{equation}
C^qK_0^{q/(\gamma+2s)} 2^{q^2/(\gamma+2s)} <  C^{q/\gamma}\left(\frac{q}{\gamma}+\frac{d}{2s}\right)^{q/\gamma}.
\end{equation}
Then
\begin{equation}
K_0 \lesssim \frac{q}{2^q} \lesssim 1,
\end{equation}
so $K_0$ is bounded by a universal constant.  

From \eqref{e:Nq2} and Theorem \ref{t:f-upper-bound}, we have 
\begin{equation}
f(t,x,v)\lesssim_q K_0^{\frac{q}{\gamma+2s}}t^{-\frac{q}{\gamma}-\frac{d}{2s}}|v|^{-q},
\end{equation}
for all $q\geq q_0$. Fix some $q >\max(q_0,d+\gamma+2s)$.  Then for any $R>0$, we have that 
\begin{equation}
\begin{split}
t^{\frac{\gamma+2s}{\gamma}}\int\limits_{\R^d} f(t,x,v)|v|^{\gamma+2s}\dd v &\leq t^{\frac{\gamma+2s}{\gamma}}\int\limits_{B_R} f(t,x,v)|v|^{\gamma+2s} \dd v + t^{\frac{\gamma+2s}{\gamma}}\int\limits_{\R^d\setminus B_R} f(t,x,v)|v|^{\gamma+2s}\dd v 
\\& \lesssim_q t^{\frac{\gamma+2s}{\gamma}}R^{\gamma+2s-2}\int\limits_{B_R} f(t,x,v)|v|^{2} \dd v + t^{\frac{\gamma+2s}{\gamma}-\frac{q}{\gamma}-\frac{d}{2s}}K_0^{\frac{q}{\gamma+2s}}\int\limits_{\R^d\setminus B_R} |v|^{-q+\gamma+2s}\dd v 
\\& \lesssim_q t^{\frac{\gamma+2s}{\gamma}}R^{\gamma+2s-2} + t^{\frac{\gamma+2s}{\gamma}-\frac{q}{\gamma}-\frac{d}{2s}}K_0^{\frac{q}{\gamma+2s}}R^{-q+d+\gamma+2s}.
\end{split}
\end{equation}

Taking $R = K_0^\alpha t^{-\beta}$, where 
\begin{equation}
\alpha = \frac{q}{(\gamma+2s)(q-d-2)} , \qquad \beta = \frac{\frac{q}{\gamma} + \frac{d}{2s}}{q-d-\gamma-2s},
\end{equation}
we get that 
\begin{equation}
t^{\frac{\gamma+2s}{\gamma}}\int\limits_{\R^d} f(t,x,v)|v|^{\gamma+2s}\dd v \lesssim_q K^{p_1(\gamma,s,d,q)} t^{p_2(\gamma,s,d,q)},
\end{equation}
where 
\[\begin{split}
p_1(\gamma,s,d,q) &:= \frac{\gamma+2s-2}{\gamma+2s}\frac{q}{q-d-2},\\
p_2(\gamma,s,d,q) &:= \frac{\gamma+2s}{\gamma} -\frac{\gamma+2s-2}{\gamma}\frac{q}{q-d-\gamma-2s} - \frac{d(\gamma+2s-2)}{2s(q-d-\gamma-2s)}.
\end{split}
\]
The key point is that as $q\to \infty$, the exponents  converge $p_1 \to \displaystyle\frac{\gamma+2s-2}{\gamma+2s} <1$ and $p_2\to \displaystyle\frac{2}{\gamma} > 0$ respectively.  Hence by taking $q$ sufficiently large (depending on $\gamma$, $s$, and $d$), we get that 
\begin{equation}
t^{\frac{\gamma+2s}{\gamma}}\int\limits_{\R^d} f(t,x,v)|v|^{\gamma+2s}dv \lesssim K_0^{1-\epsilon}
\end{equation}
for all $(t,x)\in (0,1]\times \T^d$, and some small $\epsilon\in (0,1)$.  
Taking the supremum in $(t,x)$, we thus have that 
\begin{equation}
K_0 \lesssim K_0^{1-\epsilon},
\end{equation}
which implies $K_0 \lesssim 1$. 
Thus, as $K_0$ is bounded by a universal constant, the pointwise bounds of Theorem \ref{t:f-upper-bound} are universal as well.  
\end{proof}

\section*{Acknowledgements}
We would like to thank Luis Silvestre for originally suggesting the problem and for advising the first author.  
The first author was partially supported by an NSF postdoctoral fellowship, NSF DMS 1902750.  The second author was partially supported by a Ralph E. Powe Junior Faculty Enhancement Award from ORAU.

\bibliographystyle{abbrv}
\bibliography{BoltzmannDecayForHardPotentials}

\begin{thebibliography}{10}

\bibitem{alonso2013exponential}
R.~Alonso, J.~A. Ca\~{n}izo, I.~Gamba, and C.~Mouhot.
\newblock A new approach to the creation and propagation of exponential moments
  in the {B}oltzmann equation.
\newblock {\em Comm. Partial Differential Equations}, 38(1):155--169, 2013.

\bibitem{arkeryd1983boltzmann}
L.~Arkeryd.
\newblock {$L^{\infty }$} estimates for the space-homogeneous {B}oltzmann
  equation.
\newblock {\em J. Statist. Phys.}, 31(2):347--361, 1983.

\bibitem{bobylev1997moment}
A.~V. Bobylev.
\newblock Moment inequalities for the {B}oltzmann equation and applications to
  spatially homogeneous problems.
\newblock {\em J. Statist. Phys.}, 88(5-6):1183--1214, 1997.

\bibitem{LandauUs}
S.~Cameron, L.~Silvestre, and S.~Snelson.
\newblock Global a priori estimates for the inhomogeneous {L}andau equation
  with moderately soft potentials.
\newblock {\em Annales de l'Institut Henri Poincar\'e (C) Analyse Non
  Lin\'eare}, 35(3):625--642, 2018.

\bibitem{carleman1933boltzmann}
T.~Carleman.
\newblock Sur la th\'{e}orie de l'\'{e}quation int\'{e}grodiff\'{e}rentielle de
  {B}oltzmann.
\newblock {\em Acta Math.}, 60(1):91--146, 1933.

\bibitem{desvillettes1993moments}
L.~Desvillettes.
\newblock Some applications of the method of moments for the homogeneous
  {B}oltzmann and {K}ac equations.
\newblock {\em Arch. Rational Mech. Anal.}, 123(4):387--404, 1993.

\bibitem{desvillettes2009stability}
L.~Desvillettes and C.~Mouhot.
\newblock Stability and uniqueness for the spatially homogeneous {B}oltzmann
  equation with long-range interactions.
\newblock {\em Arch. Ration. Mech. Anal.}, 193(2):227--253, 2009.

\bibitem{elmroth1983boltzmann}
T.~Elmroth.
\newblock Global boundedness of moments of solutions of the {B}oltzmann
  equation for forces of infinite range.
\newblock {\em Arch. Rational Mech. Anal.}, 82(1):1--12, 1983.

\bibitem{gamba2017exponential}
I.~Gamba, N.~Pavlovi\'c, and M.~Taskovi\'c.
\newblock On pointwise exponentially weighted estimates for the {B}oltzmann
  equation.
\newblock {\em Preprint. arXiv:1703.06448}, 2017.

\bibitem{gamba2009upper}
I.~M. Gamba, V.~Panferov, and C.~Villani.
\newblock Upper {M}axwellian bounds for the spatially homogeneous {B}oltzmann
  equation.
\newblock {\em Arch. Ration. Mech. Anal.}, 194(1):253--282, 2009.

\bibitem{gressman2011boltzmann}
P.~T. Gressman and R.~M. Strain.
\newblock Global classical solutions of the {B}oltzmann equation without
  angular cut-off.
\newblock {\em J. Amer. Math. Soc.}, 24(3):771--847, 2011.

\bibitem{gualdani2017htheorem}
M.~P. Gualdani, S.~Mischler, and C.~Mouhot.
\newblock Factorization of non-symmetric operators and exponential
  {$H$}-theorem.
\newblock {\em M\'{e}m. Soc. Math. Fr. (N.S.)}, (153):137, 2017.

\bibitem{henderson2017smoothing}
C.~Henderson and S.~Snelson.
\newblock {$C^\infty$} smoothing for weak solutions of the inhomogeneous
  {L}andau equation.
\newblock {\em Preprint. arXiv:1707.05710}, 2017.

\bibitem{ikenberry1956kinetic}
E.~Ikenberry and C.~Truesdell.
\newblock On the pressures and the flux of energy in a gas according to
  {M}axwell's kinetic theory. {I}.
\newblock {\em J. Rational Mech. Anal.}, 5:1--54, 1956.

\bibitem{LuisMain}
C.~Imbert, C.~Mouhot, and L.~Silvestre.
\newblock Decay estimates for large velocities in the {B}oltzmann equation
  without cut-off.
\newblock {\em Preprint. arXiv:1804.06135}, 2018.

\bibitem{imbert2018schauder}
C.~Imbert and L.~Silvestre.
\newblock The {S}chauder estimate for kinetic integral equations.
\newblock {\em Preprint. arXiv:1812.11870}, 2018.

\bibitem{imbert2019smooth}
C.~Imbert and L.~Silvestre.
\newblock Global regularity estimates for the {B}oltzmann equation without
  cut-off.
\newblock {\em Preprint. arXiv:1909.12729}, 2019.

\bibitem{imbert2016weak}
C.~Imbert and L.~Silvestre.
\newblock Weak {H}arnack inequality for the {B}oltzmann equation without
  cut-off.
\newblock {\em Journal of the European Mathematical Society}, to appear.

\bibitem{lu1999boltzmann}
X.~Lu.
\newblock Conservation of energy, entropy identity, and local stability for the
  spatially homogeneous {B}oltzmann equation.
\newblock {\em J. Statist. Phys.}, 96(3-4):765--796, 1999.

\bibitem{mischler2006cooling}
S.~Mischler and C.~Mouhot.
\newblock Cooling process for inelastic {B}oltzmann equations for hard spheres.
  {II}. {S}elf-similar solutions and tail behavior.
\newblock {\em J. Stat. Phys.}, 124(2-4):703--746, 2006.

\bibitem{silvestre2016boltzmann}
L.~Silvestre.
\newblock A new regularization mechanism for the {B}oltzmann equation without
  cut-off.
\newblock {\em Comm. Math. Phys.}, 348(1):69--100, 2016.

\bibitem{StanMain}
S.~Snelson.
\newblock The inhomogeneous {L}andau equation with hard potentials.
\newblock {\em Preprint. arXiv:1805.10264}, 2018.

\bibitem{wennberg1996povzner}
B.~Wennberg.
\newblock The {P}ovzner inequality and moments in the {B}oltzmann equation.
\newblock In {\em Proceedings of the {VIII} {I}nternational {C}onference on
  {W}aves and {S}tability in {C}ontinuous {M}edia, {P}art {II} ({P}alermo,
  1995)}, number 45, part II, pages 673--681, 1996.

\end{thebibliography}

\end{document}